\documentclass[12pt,a4paper,twoside]{amsart}
\usepackage{amsmath,amssymb,amsfonts,amsopn,mathrsfs,bm}
\usepackage{mathtools}
\usepackage{graphicx}
\usepackage{epstopdf}
\usepackage{cite}
\DeclareGraphicsExtensions{.png}
\usepackage{caption}
\usepackage{subcaption}
\usepackage{tikz,pgfplots}
\tikzstyle{every node}=[font=\small]

\newtheorem{theorem}{Theorem}[section]
\newtheorem{lemma}[theorem]{Lemma}

\newtheorem{proposition}[theorem]{Proposition}

\newtheorem{remark}[theorem]{Remark}

\theoremstyle{definition}

\newcommand{\hul}{H_{ul}}

\newcommand{\beqa}{\begin{eqnarray*}}
\newcommand{\eeqa}{\end{eqnarray*}}

\newcommand{\field}[1]{\mathbb{#1}}
\newcommand{\bR}{\field{R}}        
\newcommand{\bN}{\field{N}}        
\def\cS{\mathcal{S}}
\def\cD{\mathcal{D}}

\def\rd{\bR^d}

\def\rdd{{\bR^{2d}}}

\def\R{\Big))}

\def\<{\left<}
\def\>{\Big)>}

\def\mv1{M_v^1}

\def\mn{(m,n)}
\def\mn'{(m',n')}

\def\N{\mathbb{N}}
\def\R{\mathbb{R}}
\def\Ren{\mathbb{R}^d}

\def\tauhz0{\widehat{\mathcal{T}}^\hbar(z_0)}
\def\tauhz{\widehat{\mathcal{T}}^\hbar(z)}

\def\Sn2{S_{2}(L^{2}(\Ren))}
\def\S1{S_{1}(L^{2}(\Ren))}
\def\sig00{\sigma_{0,0}}

\DeclareMathOperator*{\esssup}{ess-\,sup}

\begin{document}
\begin{abstract}
\medskip We consider the time slicing approximations of Feynman path integrals, constructed via piecewice classical paths. A detailed study of the convergence in the norm operator topology, in the space $\mathcal{B}(L^2(\rd))$ of bounded operators on $L^2$, and even in finer operator topologies, was carried on by D. Fujiwara in the case of smooth potentials with an at most quadratic growth. In the present paper we show that the result about the convergence in $\mathcal{B}(L^2(\rd))$ remains valid if the potential is only assumed to have second space derivatives in the Sobolev space $H^{d+1}(\R^d)$ (locally and uniformly), uniformly in time. The proof is non-perturbative in nature, but relies on a precise short time analysis of the Hamiltonian flow at this Sobolev regularity and on the continuity in $L^2$ of certain oscillatory integral operators with non-smooth phase and amplitude.
\medskip
\end{abstract}

\title[Feynman path integrals for non smooth potentials]{On the time slicing approximation of Feynman path integrals for non-smooth potentials}

\author{Fabio Nicola}
\address{Dipartimento di Scienze Matematiche,
Politecnico di Torino, corso Duca degli Abruzzi 24, 10129 Torino,
Italy}
\email{fabio.nicola@polito.it}

\maketitle

\section{Introduction}

Consider the Schr\"odinger equation
\begin{equation}\label{equazione}
i\hbar \partial_t u=-\frac{1}{2}\hbar^2\Delta u+V(t,x)u
\end{equation}
where $0<\hbar\leq 1$ and the potential $V(t,x)$, $t\in \R$, $x\in\rd$ is a real-valued function. A fundamental problem is to find more or less explicit formulas for the solutions. As a general principle, R. Feynman in \cite{feynman,feynman-higgs} conjectured that, under reasonable conditions on $V(t,x)$, the propagator should be expressed as a certain formal integral on an infinite dimensional space of paths in configuration space. While this principle was explored and made mathematically rigorous by several authors,  here we will adopt the so-called ``time slicing'' approach as developed by D. Fujiwara and his school in
\cite{fujiwara1,fujiwara2,fujiwara3,fujiwara4,fujiwara5,fujiwara6,fujiwara7,ichinose2,kitada1,kitada2,kumanogo0,kumanogo1,kumanogo2,kumanogo3,kumanogo4,kumanogo5,kumanogo6,nicola,tsuchida,yajima}.\par
 In short, inspired by the well-known formula of the Schr\"odinger propagator of the free particle, in general one considers the parametrices $E^{(0)}(t,s)$ defined by
\begin{equation}\label{ezero}
E^{(0)}(t,s)f(x)=\frac{1}{(2\pi i (t-s) \hbar)^{d/2}} \int_{\rd} e^{i\hbar^{-1}S(t,s,x,y)} f(y)\, dy,
\end{equation}
where $S(t,s,x,y)$ is the action computed along the classical path $\gamma$ (i.e.\ the path satisfying the Euler-Lagrange equations) satisfying the boundary condition $\gamma(s)=y$, $\gamma(t)=x$ (which will be unique for $|t-s|$ small, under the assumptions below). Hence
\begin{equation}\label{azione}
S(t,s,x,y)=\int_s^t \mathcal{L}(\gamma(\tau),\dot{\gamma}(\tau),\tau)\,d\tau,
\end{equation}
where $\mathcal{L}$ is the Lagrangian of the corresponding classical system. \par
Then one considers a subdivision $\Omega:s=t_0<t_1<\ldots<t_L=t$ of the interval $[s,t]$ and the composition
\begin{equation}\label{zero0}
E^{(0)}(\Omega,t,s)=E^{(0)}(t,t_{L-1}) E^{(0)}(t_{L-1},t_{L-2})\ldots E^{(0)}(t_1,s),
\end{equation}
which should be regarded as an approximation of the path integral, exactly as the Riemann sums appear in the definition of the Riemann integral. Therefore, one wonders whether 
the operators $E^{(0)}(\Omega,t,s)$ converge to the actual propagator $U(t,s)$ as
\[
\omega(\Omega):=\sup\{t_j-t_{j-1}: j=1,\ldots,L\}\to 0.
\]
The problem of the convergence of $E^{(0)}(\Omega,t,s)$ in the {\it norm} operator topology, in the space $\mathcal{B}(L^2(\rd))$ of bounded operators in $L^2$, was first considered in \cite{fujiwara1}, where $V(t,x)$ was assumed to have partial derivatives 
\[
\partial^\alpha_x V\in L^\infty(\R\times \rd)\quad {\rm for}\ |\alpha|\geq 2.
\] Under the same hypotheses it was then proved in \cite{fujiwara2} that there is in fact convergence in a finer topology, at the level of the integral kernels of the involved operators. This suggests that the former result of convergence in $\mathcal{B}(L^2(\rd))$ could hold under weaker regularity assumptions on $V(t,x)$. The aim of this note is exactly to investigate this issue. \par
We consider the so-called Kato-Sobolev spaces $\hul^\kappa(\rd)$, $\kappa\in\bN$ (also called uniformly local Sobolev spaces), of functions $f\in L^1_{loc}(\rd)$ which belong to the usual $L^2$-based Sobolev spaces $H^\kappa(B)$ for every open ball $B\subset\rd$ of radius $1$, uniformly with respect to $B$, in the sense that $\sup_{B}\|f\|_{H^\kappa(B)}<+\infty$. In comparison with the usual Sobolev spaces $H^\kappa(\rd)$, the functions in $\hul^\kappa(\rd)$ for $\kappa>d/2$ are still bounded and continuous but need not decay at infinity; see \cite{Boulk,arsu,kato} and Section 2 below for details. \par
Now, we assume the following condition. \par\medskip
{\bf Assumption (A)} 
{\it $V(t,x)$ is a real-valued function in $L^1_{loc}(\R\times\rd)$ such that for almost every $t\in\R$ and $|\alpha|\leq 2$ the derivatives $\partial^\alpha_x V(t,x)$ exist and are continuous with respect to $x$, and moreover}
\begin{equation}\label{due}
\partial^\alpha_x V\in L^\infty(\R;\hul^{d+1}(\rd))\quad \textrm{for}\ |\alpha|=2.
\end{equation}
\par
 Under this assumption we will see that there exists $\delta>0$ such that if $0<|t-s|\leq \delta$ and $x,y\in\rd$ there exists only a classical path $\gamma$ satisfying $\gamma(s)=y$ and $\gamma(t)=x$, so that $S(t,s,x,y)$ is well defined. Also, the operators $E^{(0)}(t,s)$, initially defined on Schwartz functions, extend to bounded operators on $L^2(\rd)$, so that the composition $E^{(0)}(\Omega, t,s)$ in \eqref{zero0} is well defined too, for any subdivision $\Omega$ with $\omega(\Omega)\leq \delta$. \par
We have the following result. 
\begin{theorem}\label{mainteo}
Suppose the condition in Assumption {\rm (A)}. For every $T>0$ there exists $C=C(T)>0$ such that for $0<t-s\leq T$ and any subdivision $\Omega$ of the interval $[s,t]$ with $\omega(\Omega)\leq \delta$, and $0<\hbar\leq 1$, we have
\begin{equation}\label{tre}
\|E^{(0)}(\Omega,t,s)-U(t,s)\|_{L^2\to L^2}
\leq C \omega(\Omega)(t-s).
\end{equation}
\end{theorem}
We notice that the rate of convergence of $E^{(0)}(\Omega,t,s)$ is the same as that which appears in \cite{fujiwara1,fujiwara2} for smooth potentials.\par
Let us observe that usually non-smooth potentials, even more singular than those considered above, are successfully treated as a perturbation. For example, if $V(x)$ is a time-independent potential in $L^\infty(\R^d)$ (or in dimension $d=3$, $V\in L^2(\R^3)+L^\infty(\R^3)$) one has the Trotter formula (cf.\ \cite[Theorem X.66]{reed}, \cite{schulman})
\[
e^{-\frac{i}{\hbar}tH}=s-\lim_{n\to\infty} \Big(e^{-\frac{it}{\hbar n}H_0} e^{-\frac{it}{\hbar n}V}\Big)^n 
\]
where $H=H_0+V(x)$, $H_0=-(1/2)\hbar^2\Delta$. However, we should notice that the operators $\big(e^{-\frac{it}{\hbar n}H_0} e^{-\frac{it}{\hbar n}V}\big)^n$ are in general  different from the approximations $E^{(0)}(\Omega,t,s=0)$ in \eqref{zero0} and, most importantly, the limit above holds only in the {\it strong} operator topology (the Trotter formula is also valid with norm convergence in some nontrivial circumstances, but for parabolic equations; cf.\ \cite{ichinose5}). Instead the main point in Theorem \ref{mainteo} is the convergence in the norm operator topology. It seems that this issue has never been considered for unbounded and time-dependent non-smooth potentials.\par
We also like to mention the approach to path integrals as infinite dimensional Fresnel integrals, developed by S. Albeverio and coworkers (see\cite{albeverio,mazzucchi} and the references therein), which allows one to treat potentials of the type ``polynomial $+$ Fourier transform of a finite measure''. However, again, it does not seem that the authors consider the problem of the convergence in the norm operator topology.  \par\smallskip
The choice of the Sobolev exponent $d+1$ in \eqref{due} is justified as follows. All the known results about the continuity in $L^2(\rd)$ of oscillatory integral operators such as \eqref{ezero}, with non-smooth phase (and amplitude) seem to require that the second derivatives of the phase have, roughly speaking, at least $s$ additional derivatives both with respect to $x$ and $y$, for some $s>d/2$ (when the derivatives are meant in the $L^2$-Sobolev scale; see \cite{Boulk,ruz} and the references therein). This fact is at least heuristically related to the known counterexamples about the continuity in $L^2$ of pseudodifferential operators (see e.g.\ \cite[Remark 2]{Boulk}), since often the second derivatives of the phase play the same role in the estimates as the amplitude; cf.\ \eqref{defgn} below. Now, under the assumption (A) we will prove that $S(t,s,x,y)$  has in fact second space derivatives in $H^{d+1}_{ul}(\rdd)$, hence possessing $d/2+1/2$ additional derivatives both with respect to $x$ and $y$. One could probably refine the result considering in \eqref{due} the fractional Sobolev space $H^s_{ul}(\rd)$ for some $s>d$, but we preferred to avoid further technicalities. 
\par
In the proof we follow the same strategy as in the case of smooth potentials \cite{fujiwara1,fujiwara2} but we need, as a preliminary result, a refined short time analysis of the Hamiltonian flow in such low regularity spaces. This represents, in fact, the main technical issue of the paper and is achieved via a priori estimates using, as tools, compositions and inverse mapping theorems for Sobolev mappings \cite{campbel,inci}. Also, the continuity results for oscillatory integral operators with non-smooth phase and amplitude proved in \cite{Boulk} will play a key role.\par
We observe that, under additional space regularity for the potential, we can prove similar results for the convergence of higher order parametrices, where powers of $\hbar$ appear in the right-hand side of \eqref{tre}. This seems of particular interest in view of the semiclassical approximation and will be studied in Section 5.\par
As a final remark let us mention a different yet non-perturbative approach, where the Schr\"odinger propagator is constructed by superposition of wave packets \cite{cnr,cnr10, kt,marzuola,tataru}. That approach would be certainly worth investigating in connection with path integrals, especially for classes of potentials which locally have the regularity of a function whose Fourier transform is in $L^1(\rd)$ (cf.\ \cite{cgnr2,cnr10}); however the parametrices in that case do not have anymore the oscillatory integral form in \eqref{ezero}, and the problem of convergence would concern a different type of approximations. \par\medskip
The paper is organized as follows. In Section 2 we prove some preliminary results about the composition and the inverse of Kato-Sobolev mappings, and recall a sufficient condition for the continuity in $L^2$ of oscillatory integral operators. Section 3 is devoted to a detailed short time analysis of the Hamiltonian flow, the focus being on the Sobolev regularity issue. In Section 4 we prove Theorem \ref{mainteo}. Finally in Section 5 we refine the above result for higher order parametrices, provided $V$ itself is more regular. \par

\section{Preliminary results}
\subsection{Sobolev spaces} We adopt the usual notation $H^{\kappa}(B)=W^{\kappa,2}(B)$ and $H^\kappa(\rd)$, $\kappa\in\bN$, for the $L^2$-based Sobolev spaces on an open ball $B\subset\rd$ or on $\rd$, respectively. We also consider the fractional Sobolev spaces $H^s(\rd)$, with $s\in \R$. \par
In the sequel we will often use the Gagliardo-Nirenberg-Sobolev inequality (see e.g.\ \cite{nir})
\begin{equation}\label{gns}
\|\partial^\alpha f\|_{L^{p}(B)}\leq C \|f\|_{L^\infty(B)}^{1-|\alpha|/\kappa}\,\|f\|_{H^\kappa(B)}^{|\alpha|/\kappa}
\end{equation}
valid for $|\alpha|\leq \kappa$, $1/p=|\alpha|/(2\kappa)$, where $B$ is any open ball of radius $1$ in $\rd$, for a constant $C>0$ independent of $B$. \par

\subsection{Kato-Sobolev spaces}
For $\kappa\in\N$ we consider the space $\hul^{\kappa}(\rd)$ of functions $f$ in $L^1_{loc}(\rd)$ satisfying 
\[
\|f\|_{\hul^{\kappa}(\rd)}:=\sup_{B}\|f\|_{H^{\kappa}(B)}=\sup_{B}\sup_{|\alpha|\leq\kappa}\|\partial^\alpha f\|_{L^2(B)}<+\infty
\]
where the supremum is made on all open balls $B\subset\rd$ of radius $1$, and the derivatives are meant in the distribution sense. In fact, considering balls of any fixed radius $r>0$ would give equivalent norms. \par
 More generally, for every $s\in\R$ one defines the spaces $\hul^s(\rd)$ of temperate distributions $f\in\cS'(\rd)$ such that 
\[
\|f\|_{\hul^s(\rd)}:=\sup_{y\in\R^d}\|\chi(\cdot-y)f\|_{H^s(\rd)}<+\infty,
\]
where $\chi\in\ C^\infty_c(\rd)\setminus\{0\}$. Changing $\chi$ gives equivalent norms, and  it is also easy to see that this definition reduces to that above when $s$ is a non-negative integer.
Moreover if $s>d/2$, $\hul^s(\rd)\subset C(\rd)\cap L^\infty(\rd)$ is a Banach algebra \cite{arsu,Boulk,kato}. \par\medskip
In Assumption (A) in Introduction we used mixed norm spaces which are defined precisely as follows: if $\kappa\in\bN$ we denote by $L^\infty(\R;\hul^\kappa(\rd))$ the space of functions $f(t,x)$ in $L^1_{loc}(\R\times\rd)$ such that their distribution derivatives $\partial^\alpha_x f$ for $|\alpha|\leq\kappa$ are in $L^1_{loc}(\R\times\rd)$ and moreover 
\[
\|f\|_{L^\infty(\R;\hul^\kappa(\rd))}:=\esssup\limits_{t\in\R}\| f(t,\cdot)\|_{\hul^\kappa(\rd)}<+\infty.
\]
\par
We will need the following property about scaling.
\begin{proposition}
\label{pro0}
Let $\kappa\in\bN$. With the notation $a_t(x,y)=a(x,ty)$ for $0\leq t\leq 1$, $x,y\in\rd$, there exists a constant $C>0$ independent of $t$ and $a$ such that 
\[
\|a_t\|_{\hul^\kappa(\rdd)}\leq C \|a\|_{\hul^{\kappa+[d/2]+1}(\rdd)}.
\] 
\end{proposition}
Here $[d/2]$ stands for the integer part of $d/2$.
\begin{proof}
Let $|\alpha|+|\beta|\leq \kappa$ and $B$ be an open ball of radius $1$ in $\rd_y$. We have 
\[
\|\partial^\alpha_x\partial^\beta_y a_t(x,\cdot)\|_{L^2(B)}=t^{|\beta|-d/2}\|\partial^\alpha_x\partial^\beta_y a(x,\cdot)\|_{L^2(\tilde{B})}
\]
where $\tilde{B}=\{ty:\,y\in B\}$ is a ball of radius $t$.
Now, if $|\beta|\geq d/2$ we have $t^{|\beta|-d/2}\leq 1$ and one concludes easily by taking the norm in $L^2(B')$ of this expression, where $B'\subset \rd_x$ is any ball of radius 1. \par
If instead $|\beta|<d/2$, hence $|\beta|<m:=[d/2]+1$ we set $1/p=|\beta|/(2m)$ and we continue the above estimate, by H\"older's inequality, as 
\begin{align*}
\|\partial^\alpha_x\partial^\beta_y a_t(x,\cdot)\|_{L^2(B)}&\leq C t^{|\beta|-d/2+d(1/2-1/p)}\|\partial^\alpha_x\partial^\beta_y a(x,\cdot)\|_{L^p(\tilde{B})}\\
&= C  t^{|\beta|(1-d/(2m))}\|\partial^\alpha_x\partial^\beta_y a(x,\cdot)\|_{L^p(\tilde{B})}\\
&\leq C t^{|\beta|(1-d/(2m))}\|\partial^\alpha_x\partial^\beta_y a(x,\cdot)\|_{L^p(B'')}
\end{align*}
where $B'' \supseteq\tilde{B}$ is the ball of radius 1 with the same center as $\tilde{B}$.
By the Gagliardo-Nirenberg-Sobolev inequality \eqref{gns} this last expression is dominated by 
\[
C' t^{|\beta|(1-d/(2m))}\|\partial^\alpha_x a(x,\cdot)\|_{L^\infty(B'')}^{1-|\beta|/m}\|\partial^\alpha_x a(x,\cdot)\|_{H^m(B'')}^{|\beta|/m}
\]
which in turn is dominated (since $m>d/2$, and therefore $t^{|\beta|(1-d/(2m))}\leq 1$ and $H^m(B'')\hookrightarrow L^\infty(B'')$) by 
\[
C'' \|\partial^\alpha_x a(x,\cdot)\|_{H^m(B'')},
\]
and one concludes as above.
\end{proof}
\begin{remark}
One easily sees that a loss of $d/2$ derivatives in Proposition \ref{pro0} is inevitable. In fact, suppose that, for $|\alpha|=\kappa$ and some $r\geq0$ the following estimate holds:
\[
\|\partial^\alpha_x a_t\|_{L^2(B\times B)} \leq C \|a\|_{\hul^{\kappa+r}(\rdd)}
\]
where $B=\{x\in\rd:\ |x|<1\}$. We test this estimate with $a(x,y)=a^{(1)}(\lambda x) a^{(2)}(\lambda y)$, $\lambda=1/t\geq 1$, where $a^{(1)}(x)$, $a^{(2)}(y)$ are in $C^\infty_c(B)\setminus\{0\}$. We obtain the estimate
\[
\|\partial^\alpha \big(a^{(1)}(\lambda\cdot)\big)\|_{L^2(B)}\| a^{(2)}\|_{L^2(B)} \leq C' \|a^{(1)}(\lambda \cdot)a^{(2)}(\lambda \cdot)\|_{H^{\kappa+r}(\rdd)}
\]
and therefore as $\lambda\to+\infty$ we deduce $\lambda^{\kappa-d/2}\leq C'' \lambda^{\kappa+r-d}$, which implies $r\geq d/2$.
\end{remark}

We are particularly interested in the issue of the composition and inverse mapping theorem for Sobolev mappings. 
\begin{proposition}\label{pro1}
Let $g:\rd\to\rd$ be a globally bi-Lipschitz map, i.e.\ satisfying
\begin{equation}\label{bilip}
C_0^{-1}|x-y|\leq |g(x)-g(y)|\leq C_0 |x-y|,\quad x,y\in\rd
\end{equation}
for some constant $C_0>0$.\par Let $\kappa\in\bN$.
Then, if $f\in \hul^{\kappa}(\rd)\cap L^\infty(\rd)$ and $Dg\in\hul^{\kappa}(\rd;\R^{d\times d})$ we have $f\circ g\in \hul^{\kappa}(\rd)\cap L^\infty(\rd)$. 
\end{proposition}
Here $Dg$ stands for the Jacobian matrix of $g$.
\begin{proof}
The result is a variant of the known composition formulas for Sobolev mappings (see e.g. \cite{campbel,inci} and the references therein). We provide a short proof for the benefit of the reader, since we did not find the statement exactly in the present form in the literature.\par
 It is clear that for every open ball $B$ of radius $1$, \[
 \|f\circ g\|_{L^2(B)}\leq C \|f\circ g\|_{L^\infty(B)}\leq C \|f\circ g\|_{L^\infty(\rd)}\leq C \|f\|_{L^\infty(\rd)}.
 \]
  To estimate the derivatives of $f\circ g$ we use the chain rule (for the justification of the chain rule at this regularity we refer to \cite{campbel}). Namely, let $g= (g_1,\ldots,g_d)$. We can write $\partial^\alpha (f\circ g)$, $1\leq |\alpha|\leq \kappa$, as a linear combination of terms of the form 
\[
\partial^\sigma f (g(x))\, \partial^{\mu_1}g_{j_1}(x) \ldots \partial^{\mu_{|\sigma|}}g_{j_{|\sigma|}}(x)
\]
where $1\leq |\sigma|\leq |\alpha|$, $|\mu_1|,\ldots,|\mu_{|\sigma|}|\geq 1$, $\mu_1+\ldots+\mu_{|\sigma|}=\alpha$, and $j_1,\ldots, j_{|\sigma|}\in\{1,\ldots,d\}$. \par
Now, given an open ball $B\subset\rd$ of radius $1$, we estimate the norm in $L^2(B)$ of this expression by H\"older's inequality as
\begin{equation}\label{fact}
\|(\partial^\sigma f)\circ g\|_{L^{p_0}(B)}\|\partial^{\mu_1}g_{j_1}\|_{L^{p_1}(B)}\ldots \|\partial^{\mu_{|\sigma|}}g_{j_{|\sigma|}}\|_{L^{p_{|\sigma|}}(B)},
\end{equation}
where we choose
\[
\frac{1}{p_0}=\frac{|\sigma|}{2\kappa},\quad \frac{1}{p_j}=\frac{|\mu_j|-1}{2\kappa},\ j=1,\ldots,|\sigma|.
\]
Observe that, in fact, 
\[
\sum_{j=0}^{|\sigma|}\frac{1}{p_j}=\frac{|\sigma|}{2\kappa}+\frac{|\alpha|-|\sigma|}{2\kappa}=\frac{|\alpha|}{2\kappa}\leq\frac{1}{2}.
\]
On the other hand, by the bi-Lipschitz assumption \eqref{bilip} we have, with $\tilde{B}=g(B)$,
\[
\|(\partial^\sigma f)\circ g\|_{L^{p_0}(B)}\leq C \|\partial^\sigma f\|_{L^{p_0}(\tilde{B})}.
\]
Since $\tilde{B}=g(B)$ can be covered by a fixed number of balls $B'$ of radius $1$, and by the Gagliardo-Nirenberg-Sobolev inequality \eqref{gns} we can continue the estimate as
\begin{align*}
\|(\partial^\sigma f)\circ g\|_{L^{p_0}(B)}&\leq C'\sup_{B'}\|\partial^\sigma f\|_{L^{p_0}(B')}\\
&\leq C'' \sup_{B'}\|f\|_{L^\infty(B')}^{1-|\sigma|/\kappa}\|f\|_{H^\kappa(B')}^{|\sigma|/\kappa}\\
&\leq C'' \|f\|_{L^\infty(\rd)}^{1-|\sigma|/\kappa}\|f\|_{\hul^\kappa(\rd)}^{|\sigma|/\kappa}.
\end{align*} 
Hence we obtain
\[
\sup_B \|(\partial^\sigma f)\circ g\|_{L^{p_0}(B)}<+\infty.
\]
Similarly we estimate the other factors in \eqref{fact}: for $j=1,\ldots,|\sigma|$ we have  $|\mu_j|\geq 1$, so that by the Gagliardo-Nirenberg-Sobolev inequality we obtain
\[
\|\partial^{\mu_j}g\|_{L^{p_j}(B)}\leq C\|Dg\|_{L^\infty(B)}^{1-(|\mu_j|-1)/\kappa}\|Dg\|_{H^\kappa(B)}^{(|\mu_j|-1)/\kappa}
\]
which is uniformly bounded with respect to $B$ by the assumptions on $g$ (observe that \eqref{bilip} implies $Dg\in L^\infty$).
\end{proof}

We also recall the following inverse mapping theorem.  
\begin{proposition}\label{pro2} Let $\kappa\in\bN$. Let $g:\rd\to\rd$ be a globally bi-Lipschitz map, i.e. satisfying \eqref{bilip}, with $Dg\in \hul^{\kappa}(\rd;\R^{d\times d})$. Then $Dg^{-1}\in \hul^{\kappa}(\rd;\R^{d\times d})$ as well.
 \end{proposition}
\begin{proof}
By applying the inverse mapping theorem for Sobolev mappings as stated e.g.\ in \cite[Theorem 1.1]{campbel} (with $m=\kappa+1$ and $p=2$) one obtains that $Dg^{-1}\in H^\kappa(B)$ for every open ball $B\subset \rd$, even if $g$ were only locally bi-Lipschitz. Actually, since $g$ is globally bi-Lipschitz and $Dg$ is in $H^\kappa$ locally uniformly, an inspection of the proof of \cite[Theorem 1.1]{campbel} shows that the desired estimates are uniform with respect to the balls $B$, provided these latter have the same fixed radius. This concludes the proof.

\end{proof}
\begin{remark}
It also follows from the proof of Proposition \ref{pro2} that we have, in fact, $\|Dg^{-1}\|_{\hul^{\kappa}}\leq C$ for a constant $C$ depending only on $C_0,k,d$ and upper bounds for $\|Dg\|_{\hul^{\kappa}}$, where $C_0$ in the constant in \eqref{bilip}. A similar remark applies to Proposition \ref{pro1}. In the sequel we will apply freely these ``uniform versions'' of Propositions \ref{pro1} and \ref{pro2} without further comments. 
\end{remark}

\subsection{Oscillatory integral operators with non-smooth phase and amplitude}
Let $A$ be an oscillatory integral operator of the form 
\begin{equation}\label{eq19}
Af(x)=\int_{\rd} e^{i\lambda S(x,y)} a(x,y)f(y)\, dy,\quad \lambda\geq \lambda_0>0.
\end{equation}
Several conditions on the phase function $S(x,y)$ and the amplitude $a(x,y)$ are known for $A$ to be bounded on $L^2(\rd)$; see e.g. \cite{ruz} and the references therein. Here we recall the following result from \cite[Corollary 1]{Boulk}.
\begin{proposition}\label{pro3}
Let $A$ be the operator in \eqref{eq19}, and suppose $S$ real-valued with $\partial^\alpha S \in \hul^{s}(\rdd)$ for $|\alpha|=2$, and $a\in \hul^{s}(\rdd)$ for some $s>d$. Assume, moreover, that 
\begin{equation}\label{eq20}
\Big|{\rm det}\, \frac{\partial^2 S}{\partial x\partial y}(x,y)\Big|\geq\tilde{\delta}\quad x,y\in\rd
\end{equation}
for some $\tilde{\delta}>0$. \par
Then $A$, initially defined on Schwartz functions, extends to a bounded operator on $L^2(\rd)$. In fact there exists a constant $C>0$ independent of $\lambda\geq \lambda_0$, $\tilde{\delta}$, $S$ and $a$ such that
\[
\|A\|_{L^2\to L^2}\leq C\tilde{\delta}^{-1}\lambda^{-d/2}\exp(C\|D^2S\|_{\hul^{s}})\|a\|_{\hul^{s}}.
\]
\end{proposition}
Here $D^2S$ denotes the Hessian matrix of $S$.

\section{Short time analysis of the Hamiltonian flow}
\subsection{Sobolev regularity of the Hamiltonian flow}
In this section we will assume the following hypothesis on the potential $V(t,x)$. \par
Let $\kappa\in\bN$, $\kappa\geq d+1$.
 \par\medskip
{\bf Assumption (B)} {\it
 $V(t,x)$ is a real-valued function in $L^1_{loc}(\R\times\rd)$ such that for almost every $t\in\R$ and $|\alpha|\leq 2$ the derivatives $\partial^\alpha_x V(t,x)$ exist and are continuous with respect to $x$, and moreover}
\[
\partial^\alpha_x V\in L^\infty(\R;\hul^{\kappa}(\rd))\quad \textrm{for}\ |\alpha|=2.
\]
\par\smallskip
In particular, we will take $\kappa=d+1$ in Theorem \ref{mainteo} and higher values of $\kappa$ in Theorem \ref{mainteo2} below. \par
Denote by $(x(t,s,y,\eta),\xi(t,s,y,\eta))$, $s,t\in\R$, $y,\eta\in\rd$, the solution of the Hamiltonian system 
\begin{equation}\label{hams}
\dot{x}=\xi,\quad \dot{\xi}=-\nabla_x V(t,x)
\end{equation}
with initial condition at time $t=s$ given by $x(s,s,y,\eta)=y$, $\xi(s,s,y,\eta)=\eta$.\par Since $\hul^{\kappa}(\rd)\subset L^\infty(\rd)$ (because $\kappa\geq d+1>d/2$), for $|\alpha|=2$ we have $\partial^\alpha_x V\in L^\infty(\R\times\R^d)$, and moreover $\partial^\alpha_x V(t,x)$ is continuous with respect to $x$ for almost every $t$, so that the flow
\begin{equation}\label{chi}
(y,\eta)\mapsto (x(t,s,y,\eta),\xi(t,s,y,\eta))
\end{equation}
defines a $C^1$ canonical transformation $\rd\times\rd\to\rd\times\rd$.   Moreover, for fixed $s$, as a function of $t,y,\eta$, the function $x(t,s,y,\eta)$ is $C^1$ whereas the function $\xi(t,s,y,\eta)$ is locally Lipschitz.\par 
Concerning quantitative information, we observe that it follows at once from Gronwall's inequality that for any $T>0$ there exists a constant $C=C(T)>0$ such that 
\begin{multline}\label{eq1}
\Big\|\frac{\partial x}{\partial{y}}(t,s)\Big\|_{L^\infty(\rd_y\times\rd_\eta)}+ 
\Big\|\frac{\partial x}{\partial{\eta}}(t,s)\Big\|_{L^\infty(\rd_y\times\rd_\eta)}\\
+
\Big\|\frac{\partial \xi}{\partial{y}}(t,s)\Big\|_{L^\infty(\rd_y\times\rd_\eta)}+
\Big\|\frac{\partial \xi}{\partial{\eta}}(t,s)\Big\|_{L^\infty(\rd_y\times\rd_\eta)}\leq C,
\end{multline}
if $|t-s|\leq T$. In particular the map $(x(t,s),\xi(t,s))$ is globally Lipschitz. \par
In fact, we can prove the following result, which is finer when $|t-s|$ is small.
\begin{lemma}
For every $T>0$ there exists a constant $C=C(T)>0$ such that, for $|t-s|\leq T$, 
\begin{equation}\label{nove0}
\Big\|\frac{\partial x}{\partial{y}}(t,s)-{\rm Id}\Big\|_{L^\infty(\rd_y\times\rd_\eta)}\leq C|t-s|^2,\quad 
\Big\|\frac{\partial x}{\partial{\eta}}(t,s)\Big\|_{L^\infty(\rd_y\times\rd_\eta)}\leq C|t-s|
\end{equation}
\begin{equation}\label{nove}
\Big\|\frac{\partial \xi}{\partial{y}}(t,s)\Big\|_{L^\infty(\rd_y\times\rd_\eta)}\leq C|t-s|,\quad
\Big\|\frac{\partial \xi}{\partial{\eta}}(t,s)-{\rm Id}\Big\|_{L^\infty(\rd_y\times\rd_\eta)}\leq C|t-s|^2.
\end{equation}
\end{lemma}
\begin{proof}
From \eqref{hams} we have
\[
\begin{cases}\displaystyle 
\frac{\partial x}{\partial u}(t,s)=\frac{\partial y}{\partial u}+\int_s^t \frac{\partial \xi}{\partial u}(\tau,s)\, d\tau
\\ \displaystyle
\frac{\partial \xi}{\partial u}(t,s)=\frac{\partial \eta}{\partial u}-\int_s^t \frac{\partial^2}{\partial x\partial x } V(\tau,x(\tau,s)) \frac{\partial x}{\partial u}(\tau,s)\, d\tau,
\end{cases}
\]
where $u$ stands for $y_j$ or $\eta_j$, $j=1,\ldots,d$. Hence we obtain
\begin{equation*}
\frac{\partial x}{\partial u}(t,s)-\frac{\partial y}{\partial u}=\int_s^t \frac{\partial \xi}{\partial u}(\tau,s)\, d\tau
\end{equation*}
and 
\begin{align*}
\frac{\partial \xi}{\partial u}(t,s)-\frac{\partial \eta}{\partial u}+&\int_s^t \frac{\partial^2}{\partial x\partial x } V(\tau,x(\tau,s)) \frac{\partial y}{\partial u}\, d\tau\\
&=-\int_s^t \int_s^\tau \frac{\partial^2}{\partial x\partial x } V(\tau,x(\tau,s)) \frac{\partial \xi}{\partial u}(\sigma,s)\, d\sigma.
\end{align*}
Then, using \eqref{eq1} we obtain at once the desired estimates.
\end{proof} 

Our aim now is to prove that the flow $(x(t,s,y,\eta),\xi(t,s,y,\eta))$ has the same space regularity as $\nabla_x V$, namely that the first space derivatives of $x(t,s,y,\eta),\xi(t,s,y,\eta)$ belong to $\hul^{\kappa}(\rdd)$. In fact, the following precise asymptotic estimates of their Sobolev norm as $|t-s|\to 0$ will play a key role. 
\begin{proposition}\label{pro4}
Let $T>0$. There exists a constant $C>0$ depending only on $T$ and upper bounds for the norm of $D^2_xV$ in $L^\infty(\R;\hul^{\kappa}(\rd))$ such that, for $2\leq |\alpha|+|\beta|\leq \kappa+1$, 
\begin{equation}\label{eq2}
\|\partial^\alpha_y\partial^\beta_\eta x(t,s)\|_{L^2(B)}\leq C|t-s|^{|\beta|+2},\qquad \|\partial^\alpha_y\partial^\beta_\eta \xi(t,s)\|_{L^2(B)}\leq C|t-s|^{|\beta|+1}
\end{equation}
for every open ball $B\subset\rd_y\times\rd_\eta$ of radius $1$.
\end{proposition}
Here we used the notation $D^2_xV$ for the Hessian of $V$ with respect to $x$.
\begin{proof} 
{\it First Step: Regularization.}\par
 First of all we observe that it is sufficient to prove the desired estimates under the additional assumption that $V(t,x)$ is smooth with respect to $x$ for almost every $t\in\R$. The argument uses standard techniques, but we provide a sketch for the benefit of the reader.\par Assume that the result holds for smooth potentials, and consider smooth regularizations 
\[
V_\epsilon(t,\cdot)=V(t,\cdot)\ast \rho_\epsilon,\quad 0<\epsilon\leq 1,
\]
where $\rho_\epsilon(x)=\epsilon^{-d}\rho(\epsilon^{-1}x)$ is a standard mollifier in $\rd$, i.e.\ $\rho\in C^\infty_c (\rd)$, $\rho\geq 0$, $\int_{\rd} \rho(x)=1$.\par
 We have $\| D^2_x V_\epsilon(t,\cdot)\|_{\hul^{\kappa}}\leq C\|D^2_x V(t,\cdot)\|_{\hul^{\kappa}}$ for a constant $C$ independent of $\epsilon$, so that the corresponding solution $(x_\epsilon(t,s),\xi_\epsilon(t,s))$ will satisfy the estimates in \eqref{eq2} with a constant $C$ independent of $\epsilon$.\par 
 On the other hand, we easily verify, in this order, the following facts. \par
To simplify notation, set 
 $X=(x,\xi)$, $\mathbf{b}(t,X)=(\xi,-\nabla_x V(t,x))$, $X_\epsilon=(x_\epsilon,\xi_\epsilon)$, $\mathbf{b}_\epsilon(t,X)=(\xi,-\nabla_x V_\epsilon(t,x))$, $Y=(y,\eta)$. \par\medskip
  
 \begin{itemize}
 \item[a)] {\it $\nabla_x V(\cdot,0)\in L^1_{loc}(\R)$ and there exists a constant $C>0$ such that 
 \[
 |\nabla_x V_\epsilon(t,x)-\nabla_x V(t,0)|\leq C(1+|x|)
 \]
 for almost every $t\in\R$ and every $x\in\rd$, $0<\epsilon\leq 1$.}\par
 Indeed, for almost every $t\in\R$ we have the Taylor expansion 
\begin{equation}\label{taylor}
V(t,x)=a_0(t)+\sum_{j=1}^d a_j(t)x_j+V^{(2)}(t,x),
\end{equation}
where $V^{(2)}\in L^\infty_{loc}(\R\times\rd)$, because $D^2_x V\in L^\infty(\R\times\rd)$ and $D^2_x V(t,x)$ is continuous with respect to $x$ for almost every $t$. Since $V\in L^1_{loc}(\R\times\rd)$, we see that the function $a_0(t)+\sum_{j=1}^d a_j(t)x_j$ is in $L^1_{loc}(\R)$ for almost every $x$, which implies $a_j\in L^1_{loc}(\R)$, $j=0,\ldots,d$. In particular $\nabla_x V(\cdot,0)\in L^1_{loc}(\R)$.\par
Similarly we have 
\[
 \nabla_x V(t,x)-\nabla_x V(t,0)=\tilde{\mathbf{b}}(t,x)
 \]
 where $|\tilde{\mathbf{b}}(t,x)|\leq C|x|$ for almost every $t\in\R$ and every $x\in\rd$. Since 
 \[
  \nabla_x V_\epsilon(t,\cdot)-\nabla_x V(t,0)=\tilde{\mathbf{b}}(t,\cdot)\ast\rho_\epsilon
 \]
 the desired conclusion follows easily. 
 
 
  \item[b)] {\it The solutions $X_\epsilon(t,s,Y)$ for fixed $s\in\R$ are bounded on the compact subsets of $\R\times\rdd$, uniformly with respect to $\epsilon$.}
  \par In fact, since $X_\epsilon(t,s,Y)=(x_\epsilon(t,s,Y),\xi_\epsilon(t,s,Y)) $ is the flow of $\mathbf{b}_\epsilon(t,X)$ we have 
  \begin{multline*}
 \qquad\qquad X_\epsilon(t,s,Y)= Y+\int_s^t (0,-\nabla_xV(t,0))d\tau\\
 +\int_s^t (\xi_\epsilon(\tau,s,Y),-\nabla_x V_\epsilon(\tau,x_\epsilon(\tau,s,Y))+\nabla_x V(t,0))d\tau.
  \end{multline*}
The estimates in the previous point and Gronwall's inequality give the desired conclusion. 

\item[c)] {\it For almost every $t\in\R$, $\mathbf{b}_\epsilon(t,X)$ converges to $\mathbf{b}(t,X)$ as $\epsilon\to 0$, uniformly on the compact subsets of $\rdd$.}

 \item[d)] {\it The difference $\mathbf{b}_\epsilon(t,X)-\mathbf{b}(t,X)$ is essentially bounded on the compact subsets of $\R\times\rdd$, uniformly with respect to $\epsilon$.}\par
 In fact with the notation in the point a) we have  $\tilde{\mathbf{b}}\in L^\infty_{loc}(\R\times\rd)$.
  On the other hand
 \[
  \nabla_x V_\epsilon(t,\cdot)-\nabla_x V(t,\cdot)=\tilde{\mathbf{b}}(t,\cdot)\ast\rho_\epsilon-\tilde{\mathbf{b}}(t,\cdot)
 \]
 and the claim follows at once. 
 
  \item[e)] {\it For fixed $s\in\R$, $X_\epsilon(t,s,Y)$ converges to $X(t,s,Y)$ as $\epsilon\to 0$, uniformly on the compact subsets of $\R\times\rd$.} \par In fact it turns out that for fixed $T,R>0$, by the point b) there exists a ball $B'\subset\rdd$ where the functions $X_\epsilon(t,s,Y)$  take values for $|t|\leq T$, $|Y|\leq R$, and we have\footnote{The $L^\infty$ norm of a vector-valued function $f$ is meant as the $L^\infty$ norm of the Euclidean norm $|f(x)|$.} 
\begin{multline*}
\qquad\qquad|X_\epsilon(t,s,Y)-X(t,s,Y)|\leq\int_s^t ||\mathbf{b}_\epsilon(\tau,\cdot)-\mathbf{b}(\tau,\cdot)\|_{L^\infty(B')}\, d\tau\\
+\int_s^t \|\nabla_X\mathbf{b}(\tau,\cdot)\|_{L^\infty(\rdd)}|X_\epsilon(\tau,s,Y)-X(\tau,s,Y)|\, d\tau.
\end{multline*}
The first integral tends to zero by the dominated convergence theorem (by the points c) and d)) and one concludes by Gronwall's inequality. 

 \end{itemize}
 Hence for every open ball $B\subset \rdd$ the functions $X_\epsilon(t,s,\cdot)=(x_\epsilon(t,s),\xi_\epsilon(t,s))$ converge in $\cD'(B)$ to $X(t,s)=(x(t,s),\xi(t,s))$ as $\epsilon\to 0$, and then the desired estimates \eqref{eq2} hold\footnote{As a consequence of the Riesz representation theorem and the density of $C^\infty_c(B)$ in $L^2(B)$, if $u_n$ is a bounded sequence in $L^2(B)$, say $\|u_n\|_{L^2(B)}\leq C$, converging to a distribution $u$ in $\mathcal{D}'(B)$, it turns out $u\in L^2(B)$ and $\|u\|_{L^2(B)}\leq C$.} for $(x(t,s),\xi(t,s))$.
\par\medskip
{\it Second step: A priori estimates.}\par
We now prove the formulas \eqref{eq2} as priori estimates. In order to apply an inductive argument we have to formulate the inductive hypothesis in the following stronger form: for \[
2\leq |\alpha|+|\beta|\leq \kappa+1,\quad |t-s|\leq T,\quad 1/p=(|\alpha|+|\beta|-1)/(2\kappa),
\] we have 
\begin{equation}\label{eq3}
\|\partial^\alpha_y\partial^\beta_\eta x(t,s)\|_{L^p(B)}\leq C|t-s|^{|\beta|+2},
\end{equation}
\begin{equation}\label{eq4}
 \|\partial^\alpha_y\partial^\beta_\eta \xi(t,s)\|_{L^p(B)}\leq C|t-s|^{|\beta|+1}
\end{equation}
for a constant $C>0$ independent of $B$, as in the statement. Therefore we prove these estimates by induction on $|\alpha|+|\beta|$.\par
We have, for $2\leq|\alpha|+|\beta|\leq \kappa+1$,
\begin{equation}\label{eq4bis}
\partial^\alpha_y\partial^\beta_\eta x(t,s,y,\eta)=\int_s^t \partial^\alpha_y\partial^\beta_\eta \xi(\tau,s,y,\eta)\,d\tau
\end{equation}
and 
\begin{multline}\label{eq4ter}
\partial^\alpha_y\partial^\beta_\eta \xi(t,s,y,\eta)=-\int_s^t \sum_{j=1}^d (\partial_{x_j}\nabla_x V)(\tau,x(\tau,s)) \partial^\alpha_y\partial^\beta_\eta x_j(\tau,s,y,\eta)\,d\tau\\
-\int_s^t b_{\alpha,\beta}(\tau,s,y,\eta)\, d\tau
\end{multline}
where $b_{\alpha,\beta}(\tau,s,y,\eta)$ is a linear combination of terms of the form 
\[
(\partial^\sigma_x\nabla_x V)(\tau,x(\tau,s)) \Big( \partial^{\nu_1}_y\partial^{\mu_1}_\eta x_{j_1}(\tau,s) \Big)\ldots  \Big( \partial^{\nu_{|\sigma|}}_y\partial^{\mu_{|\sigma|}}_\eta x_{j_{|\sigma|}}(\tau,s)\Big)
\]
where $j_1,\ldots,j_{|\sigma|}\in\{1,\ldots,d\}$, $\nu_1+\ldots+\nu_{|\sigma|}=\alpha$, $\mu_1+\ldots+\mu_{|\sigma|}=\beta$, $|\nu_j|+|\mu_j|\geq 1$ for $j=1,\ldots,|\sigma|$ and  $2\leq|\sigma|\leq |\alpha|+|\beta|$.\par
We can estimate the norm in $L^p(B)$ of this expression, with \[
\frac{1}{p}=\frac{|\alpha|+|\beta|-1}{2\kappa},
\] by H\"older's inequality as
\begin{equation}\label{eq5}
\|(\partial^\sigma_x\nabla_x V)(\tau,x(\tau,s)) \|_{L^{p_0}(B)}
\|\partial^{\nu_1}_y\partial^{\mu_1}_\eta x\|_{L^{p_1}(B)}\ldots
\|\partial^{\nu_{|\sigma|}}_y\partial^{\mu_{|\sigma|}}_\eta x\|_{L^{p_{|\sigma|}}(B)},
\end{equation}
where
\[
\frac{1}{p_0}=\frac{|\sigma|-1}{2\kappa},\quad \frac{1}{p_j}=\frac{|\nu_j|+|\mu_j|-1}{2\kappa},\ j=1,\ldots,|\sigma|.
\]
Observe that 
\[
\sum_{j=0}^{|\sigma|}\frac{1}{p_j}=\frac{|\sigma|-1+|\alpha|+|\beta|-|\sigma|}{2\kappa}=\frac{1}{p}.
\]
Since the map $(x(t,s),\xi(t,s))$ is a canonical, therefore measure preserving transformation,  the first factor in \eqref{eq5} can be written as
\begin{align*}
\|(\partial^\sigma_x\nabla_x V)(\tau,x(\tau,s)) \|_{L^{p_0}(B)}&= \|(\partial^\sigma_x\nabla_x V)(\tau,\cdot)\|_{L^{p_0}(\tilde{B})}
\end{align*}
where $\tilde{B}\subset\rd_x\times\rd_\xi$ is the image of $B$ by the flow $(x(t,s),\xi(t,s))$ (we are thinking of $V(\tau,x)$ as a function of $\tau,x,\xi$, constant with respect to $\xi$). 
Now, using \eqref{eq1}, we see that $\tilde{B}$ can be covered by $N=N(T,d)$ boxes $B'\times B''$, with $B',B''\subset\rd$ balls of radius $1$, for $|t-s|\leq T$, so that we can continue the estimate, by the Gagliardo-Nirenberg-Sobolev inequality \eqref{gns} (using $2\leq |\sigma|\leq \kappa+1$) as
\begin{align*}
\|(\partial^\sigma_x\nabla_x V)(\tau,x(\tau,s)) \|_{L^{p_0}(B)}&\leq C\sup_{B''}\|\partial^\sigma_x\nabla_x V(\tau,\cdot) \|_{L^{p_0}(B'')}\\
&\leq C'\sup_{B''}\|D^2_x V\|_{L^\infty(\R\times \R^d)}^{1-(|\sigma|-1)/\kappa}\|D^2_x V\|_{L^\infty(\R;H^\kappa(B''))}^{(|\sigma|-1)/\kappa}\\
&\leq C'\|D^2_x V\|_{L^\infty(\R\times \R^d)}^{1-(|\sigma|-1)/\kappa}\|D^2_x V\|_{L^\infty(\R;\hul^{\kappa})}^{(|\sigma|-1)/\kappa}.
\end{align*}
The last expression is finite, since 
\[
\|D^2_xV\|_{L^\infty(\R\times \R^d)}\leq C \|D^2_xV\|_{L^\infty(\R;\hul^{\kappa})}<+\infty
\] by assumption. \par
To treat the other factors in \eqref{eq5} observe that 
\[
\|\partial^{\nu_j}_y\partial^{\mu_j}_\eta x(\tau,s)\|_{L^{p_j}(B)}\leq C|\tau-s|^{|\mu_j|}
\]
for a constant $C=C(T)$ if $|\tau-s|\leq T$. 
This holds for $|\nu_j|+|\mu_j|=1$ by \eqref{nove0} and if $|\nu_j|+|\mu_j|\geq 2$ by the inductive hypothesis (because $|\nu_j|+|\mu_j|\leq |\alpha|+|\beta|-1$, therefore in the latter case we have $|\alpha|+|\beta|\geq 3$). \par
In conclusion, since $\mu_1+\ldots+\mu_{|\sigma|}=\beta$ we have
\[
\|b_{\alpha,\beta}(\tau,s)\|_{L^p(B)}\leq C|\tau-s|^{|\beta|}
\]
and we can estimate in \eqref{eq4ter}
\[
\|\partial^\alpha_y\partial^\beta_\eta \xi(t,s)\|_{L^p(B)}\leq C\int_s^t \|\partial^\alpha_y\partial^\beta_\eta x(\tau,s)\|_{L^p(B)}\, d\tau+C|t-s|^{|\beta|+1}.
\]
This estimate, together with \eqref{eq4bis} and Gronwall's inequality gives \eqref{eq4} and also \eqref{eq3}.\hskip1cm\ 
\end{proof}

Consider now the map
\[
(y,\zeta)\mapsto (\tilde{x}(t,s,y,\zeta),\tilde{\xi}(t,s,y,\zeta)):=(x(t,s,y,\zeta/(t-s)),(t-s)\xi(t,s,y,\zeta/(t-s)),
\]
with $s,t\in\R$ $s\not=t$, $y,\zeta\in\rd$. Observe that it is a $C^1$ canonical transformation. \par
We have the following result. 
\begin{proposition}\label{pro5} We have, for $j,k=1,\ldots,d$, 
\begin{equation}\label{eq6-0}
\frac{\partial \tilde{x}_j}{\partial \zeta_k}=\delta_{j,k}-(t-s)^2a_{jk}(t,s,y,\zeta),
\end{equation}
\begin{equation}\label{eq6}
\frac{\partial \tilde{\xi}_j}{\partial \zeta_k}=\delta_{j,k}-(t-s)^2b_{jk}(t,s,y,\zeta),
\end{equation}
\begin{equation}\label{eq6-2}
\frac{\partial \tilde{x}_j}{\partial y_k}=\delta_{j,k}-(t-s)^2c_{jk}(t,s,y,\zeta),
\end{equation}
\begin{equation}\label{eq6-3}
\frac{\partial \tilde{\xi}_j}{\partial y_k}=(t-s)^2c'_{jk}(t,s,y,\zeta),
\end{equation}
where $a_{jk}(t,s),b_{jk}(t,s),c_{jk}(t,s),c'_{jk}(t,s)$ are functions in a bounded subset of $\hul^{\kappa}(\rdd)$, if $0<|t-s|\leq T$, for every fixed $T>0$.
\end{proposition}
\begin{proof}
Let us prove formula \eqref{eq6}; the other formulas can be deduced similarly. First of all we observe that \eqref{eq6} defines the functions $b_{jk}(t,s)$, since $s\not=t$. Now, if $B\subset\rd_y\times\rd_\zeta$ is any open ball of radius $1$ we have clearly by the second formula in \eqref{nove} that, with $\eta=\zeta/(t-s)$,
\begin{align*}
\|b_{jk}(t,s)\|_{L^2(B)}&=(t-s)^{-2}\|\frac{\partial \xi_j}{\partial \eta_k}(t,s,y,\zeta/(t-s))-\delta_{jk}\|_{L^2(B)}\\
&\leq C(t-s)^{-2}\|\frac{\partial \xi_j}{\partial \eta_k}(t,s,y,\zeta/(t-s))-\delta_{jk}\|_{L^\infty(B)}\leq C'
\end{align*}
for $0<|t-s|\leq T$ and some $C'=C'(T)>0$ independent of $B$.\par
We now estimate the derivatives of $b_{jk}$. For $1\leq |\alpha|+|\beta| \leq \kappa$ we have
\[
\partial^\alpha_y\partial^\beta_\zeta b_{jk}(t,s,y,\zeta)=-(t-s)^{-2-|\beta|}\Big(\partial^\alpha_y\partial^\beta_\eta \frac{\partial \xi_j}{\partial \eta_k}\Big)(t,s,y,\zeta/(t-s))
\] 
and therefore
\[
\|\partial^\alpha_y\partial^\beta_\zeta b_{jk}(t,s)\|_{L^2(B)}\leq |t-s|^{-2-|\beta|+d/2}\| \partial^\alpha_y\partial^\beta_\eta \frac{\partial \xi_j}{\partial \eta_k}(t,s)\|_{L^2(\tilde{B})},
\]
where 
\[
\tilde{B}=\{(y,\eta)\in\rd\times\rd:\, (y,(t-s)\eta)\in B\}.
\] Observe that $\tilde{B}$ can be covered by $N=O(|t-s|^{-d})$ balls $B'\subset\rd_y\times\rd_\eta$ of radius $1$, and therefore 
\[
\| \partial^\alpha_y\partial^\beta_\eta \frac{\partial \xi_j}{\partial \eta_k}(t,s)\|_{L^2(\tilde{B})}\leq C|t-s|^{-d/2}\sup_{B'}\| \partial^\alpha_y\partial^\beta_\eta \frac{\partial \xi_j}{\partial \eta_k}(t,s)\|_{L^2(B')}\leq C' |t-s|^{-d/2+|\beta|+2},
\]
where in the last inequality we used the second formula in \eqref{eq2}.\par
Summing up we have 
\[
\|\partial^\alpha_y\partial^\beta_\zeta b_{jk}(t,s)\|_{L^2(B)}\leq C
\]
for some $C=C(T)>0$ independent of $B$, for $0<|t-s|\leq T$.
\end{proof}

It follows from the proof, or more simply from the inclusion $\hul^{\kappa}(\rdd)\subset L^\infty(\rdd)$ (recall $\kappa\geq d+1$), that the functions $a_{jk}(t,s),b_{jk}(t,s),c_{jk}(t,s),c'_{jk}(t,s)$ belong to a bounded subset of $L^\infty(\rdd)$ for $0<|t-s|\leq T$.\par
In particular, we have obtained the following regularity result. 
\begin{proposition}\label{pro6}
For any fixed $T>0$, the components of the canonical transformation $(\tilde{x}(t,s,y,\zeta),\tilde{\xi}(t,s,y,\zeta))$, have first space derivatives in a  bounded subset of $\hul^{\kappa}(\rdd)$, when $0<|t-s|\leq T$.
\end{proposition}
The now study the regularity of the inverse function of $\zeta\mapsto \tilde{x}(t,s,y,\zeta)$, for $|t-s|$ small.
\begin{proposition}\label{pro5bis}
There exists $\delta>0$ such that for $0<|t-s|\leq\delta$ and $y\in\rd$ the map 
\[
\zeta\mapsto \tilde{x}(t,s,y,\zeta)=x(t,s,y,\zeta/(t-s))
\]
 is invertible $\rd\to\rd$ and its inverse $\zeta=\zeta(t,s,\tilde{x},y)$ has first derivatives with respect to $\tilde{x}$ and $y$ in a bounded subset of $\hul^{\kappa}(\rdd)$.\par
  More precisely, for $j,k=1,\ldots,d$, we have 
\begin{equation}\label{eq7}
\frac{\partial\zeta_j}{\partial \tilde{x}_k}(t,s,\tilde{x},y)=\delta_{jk}-(t-s)^2 d_{jk}(t,s,\tilde{x},y)
\end{equation}
and 
\begin{equation}\label{eq8}
\frac{\partial\zeta_j}{\partial y_k}(t,s,\tilde{x},y)=-\delta_{jk}-(t-s)^2 d'_{jk}(t,s,\tilde{x},y)
\end{equation}
where $d_{jk}(t,s,\tilde{x},y)$ and $d'_{jk}(t,s,\tilde{x},y)$ belong to a bounded subset of $\hul^{\kappa}(\rd_{\tilde{x}}\times\rd_y)$ for $0<|t-s|\leq\delta$.
\end{proposition}
\begin{proof}
Since $\hul^\kappa(\rdd)$ is a Banach algebra ($\kappa\geq d+1$), by \eqref{eq6-0} we have
\begin{equation}\label{eq8bisbis}
{\rm det}\, \frac{\partial \tilde{x}}{\partial\zeta}(t,s,y,\zeta)=1-(t-s)^2 a(t,s,y,\zeta)
\end{equation}
for some function $a(t,s,y,\zeta)$ in a bounded subset of $\hul^\kappa(\rdd)\subset L^\infty(\rdd)$, for $0<|t-s|\leq T$, for any $T>0$. Choose $\delta>0$ such that 
\begin{equation}\label{eq8ter}
(t-s)^2 \|a(t,s,\cdot,\cdot)\|_{L^\infty}\leq 1/2
\end{equation}
for $0<|t-s|\leq \delta$.

The invertibility of the $C^1$ map $\zeta\mapsto \tilde{x}(t,s,y,\zeta)$, for $0<|t-s|\leq \delta$, then follows from Hadamard's global inversion theorem \cite[Theorem 2, page 93]{choquet}. Moreover, for $0<|t-s|\leq \delta$ the map $(y,\zeta)\mapsto (\tilde{x}(t,s,y,\zeta),y)$ will be globally bi-Lipschitz, uniformly with respect to $t,s$, with first derivatives in a bounded subset of $\hul^\kappa(\rdd)$ (by \eqref{eq6-0}). By Proposition \ref{pro2} we see that the same holds for the inverse map $(\tilde{x},y)\mapsto (y,\zeta(t,s,\tilde{x},y))$.\par
In order to prove \eqref{eq7} and \eqref{eq8} we use the formulas
\[
\frac{\partial\zeta}{\partial \tilde{x}}(t,s,\tilde{x},y)=\Big[\frac{\partial \tilde{x}}{\partial\zeta}(t,s,y,\zeta)\Big]^{-1},\quad {\rm with}\ \zeta=\zeta(t,s,\tilde{x},y)
\]
and 
\[
\frac{\partial\zeta}{\partial y}(t,s,\tilde{x},y)=-\Big[\frac{\partial \tilde{x}}{\partial\zeta}(t,s,y,\zeta)\Big]^{-1} \frac{\partial \tilde{x}}{\partial y}(t,s,y,\zeta) ,\quad {\rm with}\ \zeta=\zeta(t,s,\tilde{x},y).
\]
Now we claim that it is sufficient to prove for the inverse matrix 
the formula
\begin{equation}\label{eq9}
\Big[\frac{\partial \tilde{x}}{\partial\zeta}(t,s,y,\zeta)\Big]^{-1}=[\delta_{jk}-(t-s)^2 d''_{jk}(t,s,y,\zeta)]_{j,k=1,\ldots,d}
\end{equation}
with $d''_{jk}(t,s,y,\zeta)$ belonging to a bounded subset of $\hul^\kappa(\rdd)$. In fact, since $\hul^\kappa(\rdd)$ is a Banach algebra, by \eqref{eq9}, \eqref{eq6-2} and the composition property in Proposition \ref{pro1} (with the map $(\tilde{x},y)\mapsto (y,\zeta(t,s,\tilde{x},y))$ playing the role of the map $g$) we see that \eqref{eq7} and \eqref{eq8} follow.\par
To prove \eqref{eq9} we observe that, by \eqref{eq8bisbis} and \eqref{eq8ter} we have 
\[
\Big({\rm det}\, \frac{\partial \tilde{x}}{\partial\zeta}(t,s,y,\zeta)\Big)^{-1}=1+(t-s)^2 a'(t,s,y,\zeta),
\]
where 
\[
a'(t,s,y,\zeta)=\frac{a(t,s,y,\zeta)}{1-(t-s)^2a(t,s,y,\zeta)}
\]
is easily seen to belong to a bounded subset of $\hul^\kappa(\rdd)$ for $0<|t-s|\leq\delta$ (one can use the chain rule and arguments similar to those in the proof of Proposition \ref{pro1}, using $\kappa\geq d+1$). This together with \eqref{eq6-0} gives \eqref{eq9} and concludes the proof.\ 

\end{proof}

\subsection{Sobolev regularity of the classical action} We close this section with an analysis of the regularity of the action $S(t,s,x,y)$ defined in \eqref{azione} for, say, $0<|t-s|\leq \delta$ where $\delta>0$ is the constant appearing in Proposition \ref{pro5bis}.\par As already observed, under Assumption (B) at the beginning of this section, for $|\alpha|=2$ we have $\partial^\alpha_x V\in L^\infty(\R\times\rd)$ and $\partial^\alpha_x V(t,x)$ is continuous with respect to $x$ for almost every $t\in\R$, which implies by the arguments in \cite[Section 2]{fujiwara1} that the for every fixed $s$, the function $S(t,s,x,y)$ is almost everywhere totally differentiable with respect to $(t,x,y)$ and satisfies the Hamilton-Jacobi equation 
\[
\partial_t S(t,s,x,y)+\frac{1}{2}|\nabla_x S(t,s,x,y)|^2+V(t,x)=0.
\]
Moreover $S(t,s,x,y)$ is of class $C^2$ with respect to $x,y$ and satisfies the equations
\begin{equation}\label{eq10}
\frac{\partial S}{\partial x_j}(t,s,x,y)=\xi_j(t,s,y,\eta(t,s,x,y))
\end{equation}
and
\begin{equation}\label{eq11}
\frac{\partial S}{\partial y_j}(t,s,x,y)=-\eta_j(t,s,x,y)
\end{equation}
for $j=1,\ldots,d$.
\begin{proposition}\label{pro7}
For $0<|t-s|\leq \delta$ we have 
\begin{equation}\label{eq12}
S(t,s,x,y)=\frac{1}{2}\frac{|x-y|^2}{t-s}+(t-s)\omega(t,s,x,y),
\end{equation}
where the functions $\partial^\alpha_x\partial^\beta_y\omega(t,s,x,y)$, for $|\alpha|+|\beta|=2$, belong to a bounded subset of $\hul^\kappa(\rdd)$. 
\end{proposition}
\begin{proof}
Consider for example the derivatives of $S$ with respect to $x$. Using the formula 
\[
\omega(t,s,x,y)=(t-s)^{-2}\Big[(t-s)S(t,s,x,y)-\frac{1}{2}|x-y|^2\Big]
\]
together with \eqref{eq10}, \eqref{eq11}, \eqref{eq6} and \eqref{eq7} we deduce easily (cf. \cite[Formula (2.12)]{fujiwara1}) that
\begin{multline*}
\frac{\partial^2}{\partial x_j\partial x_k}\omega(t,s,x,y)=-b_{jk}(t,s,y,\zeta(t,s,x,y))-d_{jk}(t,s,x,y)\\
+(t-s)^2\sum_{m=1}^d b_{jm}(t,s,y,\zeta(t,s,x,y))d_{mk}(t,s,x,y).
\end{multline*}
where the functions $b_{jk}$, $d_{jk}$ are defined in \eqref{eq6} and \eqref{eq7}.  This expression belongs to a bounded subset of $\hul^\kappa(\rdd)$, for $0<|t-s|\leq \delta$, by Propositions \ref{pro5}, \ref{pro5bis} and Proposition \ref{pro1} (with the map $(x,y)\mapsto (y,\zeta(t,s,x,y))$ playing the role of the map $g$). \par
Similarly one can treat the other second derivatives of $S$, using the formulas
\[
\frac{\partial^2}{\partial x_j\partial y_k}\omega(t,s,x,y)=(t-s)^{-2}\Big[-\frac{\partial \zeta_k}{\partial x_j}(t,s,x,y)+\delta_{jk}\Big]=d_{jk}(t,s,x,y)
\]
and
\[
\frac{\partial^2}{\partial y_j\partial y_k}\omega(t,s,x,y)=(t-s)^{-2}\Big[-\frac{\partial \zeta_j}{\partial y_k}(t,s,x,y)-\delta_{jk}\Big]=d'_{jk}(t,s,x,y)
\]
where the functions $d'_{jk}(t,s,x,y)$ are defined in \eqref{eq8}. 
\end{proof}

\section{Analysis of the parametrices and proof of Theorem \ref{mainteo}}
In this section we suppose that $V(t,s)$ satisfies the hypothesis in Assumption (A) in Introduction, which corresponds to Assumption (B) in the previous section with $\kappa=d+1$. In particular all the machinery of the previous section applies with this value of $\kappa$. \par
First of all we see that, by Proposition \ref{pro7}, the operator $E^{(0)}(t,s)$ defined in \eqref{ezero} satisfies the assumptions of Proposition \ref{pro3}, with $s=d+1$, $\lambda=\hbar^{-1}/|t-s|$, for $0<|t-s|\leq \delta$, possibly for a smaller value of $\delta$. Hence it extends to a bounded operator in $L^2(\rd)$ and verifies
\begin{equation}\label{eq20}
\|E^{(0)}(t,s)\|_{L^2\to L^2}\leq C
\end{equation}
for $0<|t-s|\leq\delta$.\par
We now show that $E^{(0)}(t,s)$ converges strongly to the identity operator as $t\to s$.  
\begin{proposition}\label{pro10}
For every $f\in L^2(\rd)$ we have 
\[
\lim_{t\to s} E^{(0)}(t,s)f=f
\]
in $L^2(\rd)$. 
\end{proposition}
\begin{proof}
First of all we observe that, by \eqref{eq20}, it is sufficient to consider $f\in \cS(\rd)$. \par The result is clearly related to the stationary phase principle, but it is not easy to justify its application at this Sobolev regularity. Instead we argue ``by density'', since we already know from
\cite[Proposition 4.3]{fujiwara1} that the result holds when $V$ satisfies
\begin{equation}\label{eq40}
\partial^\alpha_x V\in L^\infty(\R\times\rd)\quad\textit{for}\ |\alpha|\geq 2.
\end{equation} \par
 Now, let $V_\epsilon$ be smooth regularizations of $V$, as in the proof of Proposition \ref{pro4}. It is easy to see (using $\hul^{d+1}(\rd)\subset L^\infty(\rd)$ and $L^1(\rd)\ast L^\infty(\rd)\subset L^\infty(\rd)$) that, if $V$ satisfies Assumption (A) in introduction, the potentials $V_\epsilon$ enjoy the property in \eqref{eq40}.\par

 Let $S_\epsilon(t,s,x,y)$ be the corresponding generating functions and
\begin{equation}\label{ezero-20}
E^{(0)}_\epsilon(t,s)f(x):=\frac{1}{(2\pi i (t-s) \hbar)^{d/2}} \int_{\rd} e^{i\hbar^{-1}S_\epsilon(t,s,x,y)} f(y)\, dy.
\end{equation}
 As already observed in the proof of Proposition \ref{pro4}, $D^2_x V_\epsilon(t,\cdot)$ belongs to a bounded subset of $\hul^{d+1}(\rd)$, so that all the results in Section 3 hold (with $\kappa=d+1$) uniformly with respect to $\epsilon$. In particular, the functions $S_\epsilon(t,s,x,y)$ are defined for $0<|t-s|\leq \delta$ for some fixed $\delta>0$, independent of $\epsilon$. \par
 Let us now prove that $E^{(0)}_\epsilon(t,s)f(x)$ converges to $E^{(0)}(t,s)f(x)$ for every $x\in\rd$, if $0<|t-s|\leq\delta$.\par
We first claim that \[
S_\epsilon(t,s,x,y)\to S(t,s,x,y)\quad  \textrm{pointwise as $\epsilon\to 0$}
\]
if $0<|t-s|\leq\delta$. In fact we have, with obvious notation,
\[
S_\epsilon(t,s,x,y)=\int_s^t \frac{1}{2}|\xi_\epsilon(\tau,s,y,\eta_\epsilon(t,s,x,y))|^2-V_\epsilon(\tau,x_\epsilon(\tau,s,y,\eta_\epsilon(t,s,x,y)))\, d\tau.
\]
We then apply the dominated convergence theorem: to check the convergence of the integrand function for almost every $\tau\in\R$ one uses the point e) in the proof of Proposition \ref{pro4}, the fact that $\eta_\epsilon(t,s,x,y)\to \eta(t,s,x,y)$ pointwise (which follows easily by contradiction) and the fact that $V_\epsilon(\tau,\cdot)\to V(\tau,\cdot)$ uniformly on the compact subsets of $\rd$, for almost every $\tau$. To check that the integrand function is conveniently dominated for fixed $t,s,x,y$, it is useful to consider the Taylor expansion (cf.\ \eqref{taylor}) 
\[
V(t,x)=a_0(t)+\sum_{j=1}^d a_j(t)x_j+V^{(2)}(t,x),
\]
where $a_j\in L^1_{loc}(\rd)$, $j=0,\ldots,d$, and $V^{(2)}\in L^\infty_{loc}(\R\times\rd)$.
If the mollifier $\rho(x)$ is e.g.\ an even function when restricted to every coordinate axis, we have 
\[
V_\epsilon(\tau,x)=V(\tau,x)\ast \rho_\epsilon(x)=a_0(\tau)+\sum_{j=1}^d a_j(\tau)x_j+V^{(2)}(\tau,x)\ast\rho_\epsilon(x).
\]
 Hence the desired claim follows easily from the point b) in the proof of Proposition \ref{pro4}.\par
Now, by the claim just proved we have indeed
\[
E_\epsilon^{(0)}(t,s)f(x)\to E^{(0)}(t,s)f(x)
\]
for every $x\in\rd$, as $\epsilon\to 0$, by the dominated convergence theorem, because $f\in\cS(\rd)\subset L^1(\rd)$. \par
By the Fatou theorem and the definition of ``$\liminf$'', for every $\mu>0$ there exists $\epsilon_0>0$ such that
\begin{align*}
\|E^{(0)}(t,s)f-f\|_{L^2}&\leq \liminf_{\epsilon\to 0}\|E_\epsilon^{(0)}(t,s)f-f\|_{L^2}
\\ &\leq\|E_{\epsilon_0}^{(0)}(t,s)f-f\|_{L^2}+\mu.
\end{align*}
Now, we know from the analogous result for smooth potentials in \cite[Proposition 4.3]{fujiwara1} that 
\[
\lim_{t\to s}\|E_{\epsilon_0}^{(0)}(t,s)f-f\|_{L^2}=0
\]
and therefore we conclude that
\[
\limsup_{t\to s}\|E^{(0)}(t,s)f-f\|_{L^2}\leq\mu.
\]
Since $\mu$ is arbitrary, we obtain the desired conclusion. 
\end{proof}

We continue the study of $E^{(0)}(t,s)$ by observing that it is a parametrix in the sense that 
\begin{equation}\label{eq222}
\big(i\hbar\partial_t+\frac{1}{2}\hbar^2\Delta-V(t,x)\big) E^{(0)}(t,s)f=G^{(0)}(t,s)f
\end{equation}
with
\begin{equation}\label{defgn}
G^{(0)}(t,s)f =\frac{(1/2)i\hbar(t-s)}{(2\pi i (t-s) \hbar)^{d/2}} \int_{\rd} e^{i\hbar^{-1}S(t,s,x,y)} \Delta_x \omega (t,s,x,y) f(y)\, dy,
\end{equation}
cf.\ \cite[Formula (1.12)]{fujiwara2}, where $\omega(t,s,x,y)$ is defined in \eqref{eq12} (differentiation under the integral sign is justified e.g.\ for $f\in\cS(\rd)$ exactly as in \cite[Proposition 4.5]{fujiwara1}).\par
It follows again from Propositions \ref{pro7} (with $\kappa=d+1$) and Proposition \ref{pro3} (with $s=d+1$)  that $G^0(t,s)$ is bounded on $L^2(\rd)$ and satisfies the key estimate
\begin{equation}\label{eq13}
\| G^{(0)}(t,s)\|_{L^2\to L^2}\leq C\hbar |t-s|
\end{equation}
for $0<|t-s|\leq\delta$, possibly for a smaller value of $\delta$. \par
Concerning the actual propagator, we can easily prove its existence for more general potentials. 
\begin{proposition}\label{prop8}
Let $V(t,x)$ be a real-valued function in $L^1_{loc}(\R\times\rd)$ such that for almost every $t\in\R$ and $|\alpha|\leq 2$ the derivatives $\partial^\alpha_x V(t,x)$ exist and are continuous with respect to $x$, with $\partial^\alpha_x V\in L^\infty(\R\times\rd)$ for $|\alpha|=2$.\par Let $s\in\R$.
Then the Cauchy problem 
\[
\begin{cases}
i\hbar \partial_t u=-\frac{1}{2}\hbar^2\Delta u+V(t,x)u\\
u(s,x)=u_0(x)
\end{cases}
\]
is forward and backward globally wellposed in $L^2(\rd)$. \par
If we denote by $U(t,s)$ the corresponding propagator, for every $T>0$ there exists a constant $C=C(T)>0$ independent of $\hbar$ ($0<\hbar\leq 1$), such that
\begin{equation}\label{eq18}
\|U(t,s)\|_{L^2\to L^2}\leq C
\end{equation}
for $|t-s|\leq T$. 
\end{proposition}
\begin{proof}
Let $\chi(\xi)$ be a smooth function in $\rd$ supported in $|\xi|\leq 2$, with $\chi(\xi)=1$ for $|\xi|\leq 1$. We split the potential as 
\begin{equation}\label{eq16}
V(t,x)=\underbrace{\chi(D_x)V(t,x)}_{V_0(t,x)}+\underbrace{(1-\chi(D_x))V(t,x)}_{V_1(t,x)}.
\end{equation}
We claim that 
\begin{equation}\label{eq14}
\partial^\alpha_x V_0\in L^\infty(\R\times\rd)\quad{\rm for}\ |\alpha|\geq 2
\end{equation}
and 
\begin{equation}\label{eq15}
V_1\in L^\infty(\R\times\rd).
\end{equation}
Let us prove \eqref{eq14}. For $|\alpha|\geq 2$, write $\alpha=\beta+\gamma$, with $|\gamma|=2$ and 
\[
\partial^\alpha_x V_0(t,x)=\chi(D_x)\partial^\beta_x(\partial^\gamma_x V(t,x)).
\] We have $\partial^\gamma_x V\in L^\infty(\R\times\rd)$ by assumption and the Fourier multiplier $\chi(D_x)\partial^\beta_x$ is bounded in $L^\infty(\rd)$, because its symbol is the smooth compactly supported function $i^{|\beta|}\chi(\xi)\xi^\beta$. \par
Concerning \eqref{eq15}, we write 
\[
V_1(t,x)=(1-\chi(D_x))\Delta^{-1}_x\Delta_x V(t,x).
\]
By assumption $\Delta_x V\in L^\infty(\R\times\rd)$ and the Fourier multiplier $(1-\chi(D_x))\Delta_x^{-1}$ is bounded on $L^\infty(\rd)$, because its symbol $\sigma(\xi)=-|\xi|^{-2}(1-\chi(\xi))$ is smooth in $\rd$ and has (inverse) Fourier transform in $L^1(\rd)$; indeed repeated integrations by parts give $|x|^{k}\widehat{\sigma}(x)\in L^\infty$ for every $k\geq d-1$.\par
We now observe that the potential $V_0(t,x)$ falls in the class considered in \cite{fujiwara1,fujiwara2}, and the corresponding propagator $U_0(t,x)$ was proved in \cite[Theorem 3.1]{fujiwara2} to satisfy the estimates
\begin{equation}\label{eq17}
\|U_0(t,s)\|_{L^2\to L^2}\leq C
\end{equation}
for $|t-s|\leq T$.  \par
The perturbed original Cauchy problem can be written in integral form as
\[
u(t)=U_0(t,s)u_0-i\hbar^{-1}\int_s^t U_0(t,\tau) V_1(\tau,\cdot)u(\tau)\, d\tau.
\]
This is a Volterra-type integral equation, which by \eqref{eq15} and \eqref{eq17} is easily seen to have a unique solution in $C([s-T,s+T];L^2(\rd))$, for every $T>0$. By \eqref{eq17} it is also clear that the propagator $U(t,s)$ satisfies \eqref{eq18}. 
\end{proof}

We are ready to prove our main result.
\begin{proof}[Proof of Theorem \ref{mainteo}]
Consider the operator  $R^{(0)}(t,s)$ defined by
\[
R^{(0)}(t,s)=E^{(0)}(t,s)-U(t,s).
\]
By Proposition \ref{pro10} and \eqref{eq222} we can write
\[
R^{(0)}(t,s)f=-i\hbar^{-1}\int_s^t U(t,\tau)G^{(0)}(\tau,s)f\,d\tau.
\]
 
Using \eqref{eq13} and \eqref{eq18} we deduce that  
\begin{equation}\label{chiave}
\|R^{(0)}(t,s)\|_{L^2\to L^2}\leq  C(t-s)^{2}
\end{equation}
for $0<|t-s|\leq\delta$.\par
 We then proceed as in \cite[Lemma 3.2]{fujiwara2}, where the analogous result for smooth potentials was proved.\par
 Namely, we have to estimate the $L^2\to L^2$ norm of the operator
\begin{align*}
E^{(0)}&(\Omega,t,s)-U(t,s)=E^{(0)}(t,t_{L-1}) E^{(0)}(t_{L-1},t_{L-2})\ldots E^{(0)}(t_1,s)\\
&=\big(U(t,t_{L-1})+R^{(0)}(t,t_{L-1})\big)\ldots \big(U(t_1,s)+R^{(0)}(t_1,s)\big)-U(t,s).
\end{align*}
We can expand the right-hand side, grouping together the consecutive factors ``of type $U$'' (using the evolution property of the propagator), and then estimate the $L^2\to L^2$ norm of each ordered product by \eqref{eq18} and \eqref{chiave}. After that it is an algebraic matter to handle the sum which arises and express the result in the desired form. For full details we refer to the proof of \cite[Lemma 3.2]{fujiwara2}, which can be repeated essentially verbatim in our framework (the seminorms $\|\cdot\|_m$ which appear there have to be replaced by the norm of bounded operators on $L^2(\rd)$). This gives the desired formula \eqref{tre}.\ \
\end{proof}

\section{The case of higher order parametrices}
In this section we assume higher regularity on the potential and we prove a stronger convergence theorem. Namely, let $N\in\bN$, $N\geq1$. Suppose that
 \par\medskip
{\it $V(t,x)$  satisfies Assumption (B) in Section 3 with $\kappa=d+1+N([d/2]+3)$.}
\par\medskip\noindent
Therefore we will apply the results of Section 3 for this value of $\kappa$.\par
Let $\delta$ be the constant appearing in Proposition \ref{pro5bis}, so that the function $S(t,s,x,y)$ is well defined for $0<|t-s|\leq \delta$.\par 
In order to construct higher order parametrices, consider the functions $a_j(t,s,x,y)$, $j=1,2,\ldots,N,$ defined by the formulas (cf.\ \cite[Formulas (3.4) and (3.5)]{fujiwara1})
\begin{equation}\label{eq21}
a_1(t,s,x,y)=\exp\Big(-\frac{1}{2}\int_s^t (\tau-s)\Delta_x \omega(\tau,s,x(\tau),y)\,d\tau  \Big)
\end{equation}
and 
\begin{equation}\label{eq22}
a_j(t,s,x,y)=-\frac{1}{2}a_1(t,s,x,y)\int_s^t\frac{\Delta_x a_{j-1}(\tau,s,x(\tau),y)}{ a_1(\tau,s,x(\tau),y)}\, d\tau,
\end{equation}
for $j=2,\ldots,N$, where $\omega(t,s,x,y)$ is the function defined in \eqref{eq12} and $x(\tau)=x(\tau,s,y,\eta(t,s,x,y))$. \par
We need the following result on the regularity of the functions $a_j(t,s,x,y)$.
\begin{proposition}\label{pro9}
For $0<|t-s|\leq \delta$ we have 
\begin{equation}\label{eq23}
a_1(t,s,x,y)=1-(t-s)^2r(t,s,x,y)
\end{equation}
and 
\begin{equation}\label{eq24}
1/a_1(t,s,x,y)=1+(t-s)^2r'(t,s,x,y),
\end{equation}
where the functions $r(t,s,x,y)$ and $r'(t,s,x,y)$ belong to a bounded subset of\\  $\hul^{d+3+(N-1)([d/2]+3)}(\rdd)$.\par
Moreover for $2\leq j\leq N$ we have 
\begin{equation}\label{eq25}
\|a_j(t,s,\cdot,\cdot)\|_{\hul^{d+3+(N-j)([d/2]+3)}(\rdd)}\leq C|t-s|^{j+1}
\end{equation}
for some constant $C>0.$
\end{proposition}
\begin{proof}
Let us now prove \eqref{eq23}. It follows from Proposition \ref{pro7} (with $\kappa=d+1+N([d/2]+3)$) that $\Delta_x\omega(\tau,s,x,y)$ belongs to a bounded subset of $\hul^{d+1+N([d/2]+3)}(\rdd)$ for $0<|\tau-s|\leq \delta$, so that by Proposition \ref{pro1} (with the map  $(y,\zeta)\mapsto (\tilde{x}(\tau,s,y,\zeta),y)$ in place of $g$) we deduce that the same is true for 
\[
\Delta_x\omega(\tau,s,\tilde{x}(\tau,s,y,\zeta),y)
\]
(the function $\tilde{x}(\tau,s,y,\zeta)$ is defined in Proposition \ref{pro5bis}). \par
By Proposition \ref{pro0} we see that the functions  
\[
\Delta_x\omega(\tau,s,\tilde{x}(\tau,s,y,[(\tau-s)/(t-s)]\zeta),y)
\]
for $0<|\tau-s|\leq |t-s|\leq \delta$ belong to a bounded subset of 
\[
\hul^{d+1+N([d/2]+3)-([d/2]+1)}(\rdd)=\hul^{d+3+(N-1)([d/2]+3)}(\rdd),
\]
 and again by Proposition \ref{pro1} (with the map $(x,y)\mapsto (y,\zeta(t,s,x,y))$ in place of $g$) the same holds for 
\[
\Delta_x\omega(\tau,s,\underbrace{\tilde{x}(\tau,s,y,[(\tau-s)/(t-s)]\zeta(t,s,x,y)}_{=x(\tau)},y)=\Delta_x\omega(\tau,s,x(\tau),y).
\]
 Hence we obtain the estimate
\[
\Big\|\int_s^t(\tau-s) \Delta_x\omega(\tau,s,x(\tau),y)\Big\|_{\hul^{d+3+(N-1)([d/2]+3)}}\leq C|t-s|^2.
\]
By expanding the exponential in \eqref{eq21} as a power series in the Banach algebra $\hul^{d+3+(N-1)([d/2]+3)}(\rdd)$ we deduce \eqref{eq23}. The same arguments apply to $1/a_1(t,s,x,y)$, which gives \eqref{eq24}.\par
Finally, by the same arguments as above and induction on $j$ one sees easily that the functions $a_j(t,s,x,y)$ for $2\leq j\leq N$ verify the estimates in \eqref{eq25} (as a basic step for $j=2$ one uses that $\Delta_x a_1(\tau,s,x,y)$ belongs to a bounded subset of $\hul^{d+1+(N-1)([d/2]+3)}(\rdd)$, therefore $\Delta_x a_1(\tau,s,x(\tau),y)$ belongs to a bounded subset of $\hul^{d+1+(N-1)([d/2]+3)-([d/2]+1)}(\rdd)=\hul^{d+3+(N-2)([d/2]+3)}(\rdd)$, etc.).
\end{proof}

We now define the amplitude
\[
a^{(N)}(t,s,x,y)=\sum_{j=1}^N (i\hbar^{-1})^{1-j} a_j(t,s,x,y),\quad N=1,2,\ldots
\]
and the corresponding oscillatory integral operator
\[
E^{(N)}(t,s)f(x)=\frac{1}{(2\pi i (t-s) \hbar)^{d/2}} \int_{\rd} e^{i\hbar^{-1}S(t,s,x,y)}a^{(N)}(t,s,x,y) f(y)\, dy.
\]

It follows from Propositions \ref{pro7}, \ref{pro9} and \ref{pro3} that $E^{(N)}(t,s)$ is bounded in $L^2(\rd)$ and satisfies the estimate
\begin{equation}\label{eq26}
\|E^{(N)}(t,s)\|_{L^2\to L^2}\leq C
\end{equation}
for some constant $C>0$, for $0<|t-s|\leq \delta$, possibly for a smaller value of $\delta$. Moreover, by Propositions \ref{pro9} and \ref{pro3} we have 
\[
\|E^{(N)}(t,s)-E^{(0)}(t,s)\|_{L^2\to L^2}\leq C(t-s)^2
\]
so that, by Proposition \ref{pro10}, for every $f\in L^2(\rd)$ we have 
\[
\lim_{t\to s} E^{(N)}(t,s)f=f
\]
in $L^2(\rd)$. 
\par
On the other hand we also have (cf.\ \cite[Formulas (3.9) and (3.11)]{fujiwara1})
\[
\big(i\hbar\partial_t+\frac{1}{2}\hbar^2\Delta-V(t,x)\big) E^{(N)}(t,s)f=G^{(N)}(t,s)f
\]
with
\begin{equation*}\label{defgn-bis}
G^{(N)}(t,s)f =-\frac{(-i\hbar)^{N+1}/2}{(2\pi i (t-s) \hbar)^{d/2}} \int_{\rd} e^{i\hbar^{-1}S(t,s,x,y)} \Delta_x a_N(t,s,x,y) (t,s,x,y) f(y)\, dy.
\end{equation*}
By \eqref{eq23} (if $N=1$) or \eqref{eq25} (if $N\geq2$) we see that the amplitude $\Delta_x a_N$ satisfies the estimate
\[
\|\Delta_x a_N(t,s,\cdot,\cdot)\|_{\hul^{d+1}(\rdd)}\leq C|t-s|^{N+1}.
\]
Hence Proposition \ref{pro3} gives
\begin{equation}\label{eq28}
\|G^{(N)}(t,s)\|_{L^2\to L^2}\leq C\hbar^{N+1}|t-s|^{N+1}.
\end{equation} 
We then consider the remainder operator
\[
R^{(N)}(t,s)f:=E^{(N)}(t,s)f-U(t,s)f=-i\hbar^{-1}\int_s^t U(t,\tau)G^{(N)}(\tau,s)f\,d\tau.
\]
Using \eqref{eq28} and \eqref{eq18} we deduce that  
\begin{equation}\label{chiave2}
\|R^{(N)}(t,s)\|_{L^2\to L^2}\leq  C\hbar^{N}|t-s|^{N+2}
\end{equation}
for $0<|t-s|\leq\delta$.\par
Consider now the composition
\[
E^{(N)}(\Omega,t,s)=E^{(N)}(t,t_{L-1}) E^{(N)}(t_{L-1},t_{L-2})\ldots E^{(N)}(t_1,s),
\]
for any subdivision $\Omega:s=t_0<t_1<\ldots<t_L=t$ of the interval $[s,t]$ such that $\omega(\Omega)\leq\delta$.\par
We have the following result.
\begin{theorem}\label{mainteo2}
Assume that $V(t,x)$  satisfies Assumption (B) in Section 3 with $\kappa=d+1+N([d/2]+3)$, for some $N\geq 1$.\par For every $T>0$ there exists $C=C(T)>0$ such that for $0<t-s\leq T$ and any subdivision $\Omega$ of the interval $[s,t]$ with $\omega(\Omega)\leq \delta$, and $0<\hbar\leq 1$, we have
\begin{equation}\label{tre-bis}
\|E^{(N)}(\Omega,t,s)-U(t,s)\|_{L^2\to L^2}
\leq C \hbar^{N}\omega(\Omega)^{N+1}(t-s).
\end{equation}
\end{theorem}

\begin{proof}
As in the case considered in Section 4 (which essentially corresponds to $N=0$), the desired result follows exactly as in \cite[Lemma 3.2]{fujiwara2}; namely one expands the right-hand side of
\[
E^{(N)}(\Omega,t,s)-U(t,s)=\big(U(t,t_{L-1})+R^{(N)}(t,t_{L-1})\big)\ldots \big(U(t_1,s)+R^{(N)}(t_1,s)\big)-U(t,s),
\]
and estimates the $L^2\to L^2$ norm of each factor. To this end one uses \eqref{eq18} for the factors ``of type $U$'' and \eqref{chiave2} for the factors ``of type $R$''. We refer to \cite[Lemma 3.2]{fujiwara2} for full details.
\end{proof}


\section*{Acknowledgment}
We would like to thank Luigi Ambrosio, Ubertino Battisti, Elena Cordero and Paolo Tilli for useful discussions related to the subject of this paper.

\end{document}